\pdfoutput=1
\documentclass[12pt]{article}
\usepackage[utf8]{inputenc}
\usepackage{color}
\usepackage[dvipsnames]{xcolor}
\usepackage{amsmath}
\usepackage{parskip}
\usepackage[showframe=false]{geometry}
\usepackage{changepage}
\usepackage{graphicx}
\usepackage[refpage]{nomencl}
\graphicspath{ {images/} }
\usepackage{amsthm}
\usepackage[compact]{titlesec}         
\titlespacing{\section}{10pt}{10pt}{10pt} 
\AtBeginDocument{
  \setlength\abovedisplayskip{10pt}
  \setlength\belowdisplayskip{10pt}}
\usepackage{tikz}
\usetikzlibrary{positioning}
\usepackage{amssymb}
\usepackage{tikz}
\usepackage{mathtools}

\newcommand{\tpmod}[1]{{\@displayfalse\pmod{#1}}}

\usepackage[shortlabels]{enumitem}

\usepackage[english]{babel}
\usepackage{soul}
\newtheorem{theorem}{Theorem}[section]

\newtheorem{lemma}[theorem]{Lemma}
\newtheorem{definition}[theorem]{Definition}

\theoremstyle{definition}
\newtheorem{personalEnvironment}[theorem]{$\boldsymbol{(}\mkern-3mu$}

\newcommand{\Sym}[1]{\mathrm{Sym}({#1})}
\newcommand{\Dih}[1]{\mathrm{Dih}({#1})}

\newcommand{\SL}{\mathrm{SL}}
\renewcommand{\O}{\mathrm{O}}
\newcommand{\B}{\mathrm{B}}

\usepackage{graphicx}
\title{Two Families of Unravelled Abstract Regular Polytopes in $\B_n$}
\date{2020}
\author{Robert Nicolaides and Peter Rowley}
\begin{document}
\maketitle
\begin{abstract}

This paper exhibits two families of unravelled abstract regular polytopes in Coxeter groups of type $\B_n$. For one family they have rank 4 while the other family has arbitrarily large rank.
\end{abstract}
\section{Introduction}

Unravelled abstract regular polytopes were first introduced and studied in \cite{UNravlledPaper1}. In that paper, among other things, certain unravelled polytopes for $\SL_3(q) \rtimes \langle t \rangle,$ $t$ being the transpose inverse automorphism of $\SL_3(q)$, were analyzed. Here we continue investigating unravelled polytopes this time in Coxeter groups of type $\B_n$.

The study of abstract regular polytopes may be carried out entirely in a group theoretic environment. The parallel group theoretic concept is that of a C-string. See \cite{ARP} for details of this reformulation. For a finite group $G$ and $\{ t_1, \ldots , t_n\}$ a set of involutions is $G$, set $I = \{1, \ldots ,n\}$ and for $J \subseteq I$ $G_J = \langle t_j | j \in J \rangle$, with $G_{\emptyset} = 1$ by convention. If $J = \{j_1, \ldots , j_k\}$, then $G_J$ is sometimes written as $G_{j_1 \ldots  j_k}$. A C-string of rank $n$ for $G$ is a set of involutions $\{t_1, \ldots , t_n\}$ in $G$ such that \begin{enumerate}[(i)]
    \item $G = \langle t_1, \ldots t_n \rangle$
    \item $t_i$ and $t_j$ commute whenever $| i - j| \ge 2;$
    and 
    \item for $J,K \subseteq I,$ $G_J \cap G_K = G_{J \cap K}.$
\end{enumerate}
We can now define unravelled C-strings and unravelled abstract regular polytopes.

\begin{definition}\label{def1.1}
Suppose $G$ is a finite group with $\{t_1,\ldots,t_n\}$ a C-string for $G$ of rank $n$.
\begin{enumerate}[$(i)$]
    \item For $N \trianglelefteq G,$ set $\overline{G} = G/N$ and use $\overline{g}$ for the image of $g$ in $\overline{G}$, $g \in G$. If either $|\{\overline{t_1}, \ldots, \overline{t_n}\}| < n$ or, $|\{\overline{t_1}, \ldots, \overline{t_n}\}| = n$ and $\{\overline{t_1}, \ldots, \overline{t_n}\}$ is not a C-string for $\overline{G},$ we say that  $\{t_1,\ldots,t_n\}$ is an $N$-unravelled C-string for $G$.
    \item If $\{t_1,\ldots,t_n\}$ is an $N$-unravelled 
    C-string for all non-trivial proper normal 
    subgroups $N$ of $G$, we call $\{t_1,\ldots,t_n\}$ 
    an unravelled C-string for $G$, and refer to the 
    associated abstract regular polytope as being 
    unravelled. 
\end{enumerate}
\end{definition}

Our first theorem, just as for those on $\SL_3(q) \rtimes \langle t \rangle,$ yields unravelled C-strings of rank 4.

\begin{theorem}\label{thm1.2}
Suppose that $G=\B_n$ where $n$ is odd and $n \ge 5$. Then $G$ has a rank 4 C-string $\{t_1,t_2,t_3,t_4\}$ with Schlafli symbol 

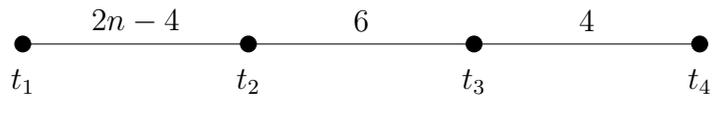
\begin{center}
    
\tikzstyle{int}=[draw, fill=black!20, minimum size=2em]
\tikzstyle{init} = [pin edge={to-,thin,black}]
\begin{tikzpicture}[baseline={(0,-1.5)}]

\draw (3,-0.5) -- (6,-0.5);
\draw (6,-0.5) -- (9,-0.5);
\draw (9,-0.5) -- (12,-0.5);
\filldraw [black] (3,-0.5) circle (3pt);
\filldraw [black] (6,-0.5) circle (3pt);
\filldraw [black] (9,-0.5) circle (3pt);
\filldraw [black] (12,-0.5) circle (3pt);
\node (p1) at (4.5,-0.2) {$2n-4$};
\node (p2) at (7.5,-0.2) {$6$};
\node (p3) at (10.5,-0.2) {$4$};

\node (t1) at (3,-1) {$t_1$};
\node (t2) at (6,-1) {$t_2$};
\node (t3) at (9,-1) {$t_3$};
\node (t4) at (12,-1) {$t_4$};
\end{tikzpicture},
\end{center}
which is unravelled when $n > 5.$ Further $G_{123} \cong \Sym{n}$ and $G_{234} \cong \mathbb{Z}_2 \times \Sym{5}.$
\end{theorem}

Our second theorem, by contrast, gives examples in $\B_n$ of arbitrary large rank.

\begin{theorem}\label{thm1.3}
Suppose that $G=\B_n$ where $n \ge 8,$ and set $m=n-4.$ Then $G$ has a rank $m$ C-string $\{t_1,t_2, \ldots,t_m\}$ with Schlafli symbol

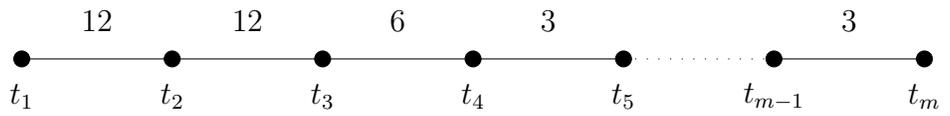
\tikzstyle{int}=[draw, fill=black!20, minimum size=2em]
\tikzstyle{init} = [pin edge={to-,thin,black}]

\begin{center}
\begin{tikzpicture}

\draw (0,0) -- (2,0);
\draw (2,0) -- (4,0);
\draw (4,0) -- (6,0);
\draw (6,0) -- (8,0);
\draw[loosely dotted] (8,0) -- (10,0);
\draw (10,0) -- (12,0);
\filldraw [black] (0,0) circle (3pt);
\filldraw [black] (2,0) circle (3pt);
\filldraw [black] (4,0) circle (3pt);
\filldraw [black] (6,0) circle (3pt);
\filldraw [black] (8,0) circle (3pt);
\filldraw [black] (10,0) circle (3pt);
\filldraw [black] (12,0) circle (3pt);
\node (p1) at (1,0.5) {$12$};
\node (p2) at (1+2,0.5) {$12$};
\node (p3) at (1+4,0.5) {$6$};
\node (p4) at (1+6,0.5) {$3$};;
\node (p6) at (1+8,0.5) {};
\node (p7) at (1+10,0.5) {$3$};
\node (t1) at (0,-0.5) {$t_1$};
\node (t2) at (2,-0.5) {$t_2$};
\node (t3) at (4,-0.5) {$t_3$};
\node (t4) at (6,-0.5) {$t_4$};
\node (t5) at (8,-0.5) {$t_5$};
\node (t6) at (10,-0.5) {$t_{m-1}$};
\node (t6) at (12,-0.5) {$t_m$};
\end{tikzpicture}.
\end{center}
Further, when $n$ is even, this C-string is unravelled.
\end{theorem}
The proofs of Theorems \ref{thm1.2} and \ref{thm1.3} occupy Sections 3 and 4 respectively, while in Section 2 we cover preliminary results for these proofs.

\section{Preliminary Results}
This short section contains the results we need in the two following sections. First of these results is one which identifies the Coxeter groups of type $\B_n$. For a set $\Omega = \{1,\ldots, n\},$ $\Sym{n} = \Sym{\Omega}$ denotes the symmetric group of degree $n$ defined on $\Omega$.

\begin{lemma}\label{lemma2.1}
Suppose that $\Omega = \{1,2\ldots,n,n+1, \ldots,2n\}.$ Let $\beta_0 = (1,n+1)$ and $\beta_i = (i,i+1)(n+i,n+i+1)$ for $1 \le i < n.$ Then $\langle \beta_0,\beta_1,\ldots,\beta_{n-1}\rangle$ is isomorphic to $\B_n$.
\end{lemma}
\begin{proof}
See (2.10) of \cite{hump}.
\end{proof}
In a similar vein to Lemma \ref{lemma2.1}, we have the well-known characterization of $\Sym{n}$.

\begin{lemma}\label{lemma2.2}
Suppose that $H$ is a group with presentation $\langle r_1,\ldots,r_{n-1} | (r_ir_j)^{m_{ij}}) \rangle$. If $m_{ii} = 1$ for $i=1,\ldots,n-1,m_{ij} =3$ if $|i-j|=1$ and $m_{ij}=2$ if $|i-j|>1,$ then $H \cong \Sym{n}$.
\end{lemma}
\begin{proof}
See (6.4) of \cite{hump}.
\end{proof}
We make frequent use of the next lemma to verify that particular sets of involutions are C-strings.

\begin{lemma}\label{lemma2.3}
Let $G = \langle t_1,\ldots t_{n}\rangle$ where for all $1\le i,j \le n $ we have $|i-j|\ge 2$ implies that $(t_it_j)^2 = 1$ and each $t_i$ an involution. If $\{t_1,\ldots, t_{n-1}\}$ and $\{t_2,\ldots,t_n\}$ are C-strings for $\langle t_1,\ldots, t_{n-1}\rangle$ and $\langle t_2,\ldots, t_{n} \rangle$, then $\{t_1,\ldots, t_{n}\}$ is a C-string for G if 
$$\langle t_1,\ldots, t_{n-1}\rangle \cap \langle t_2,\ldots, t_{n} \rangle = \langle t_2,\ldots, t_{n-1} \rangle. $$ 
\end{lemma}
\begin{proof}
See 2E16(a) in \cite{ARP}. 
\end{proof}

Our group theoretic notation is standard, as given for example in \cite{suzuki}, with the addition that $\Dih{n}$ denotes the dihedral group of order $n$.

\section{Rank 4 Unravelled C-strings}
Here we establish Theorem \ref{thm1.2}. So we are assuming that $n$ is odd and $n \ge 5$.We shall construct the C-string for $\B_n$ working in $\Sym{2n}.$ First we define the involutions $t_i$, $i =1,2,3,4$ in $\Sym{2n} = \Sym{\Omega}$, where $\Omega = \{1,\ldots,2n\}.$
\begin{personalEnvironment}$\boldsymbol{\mkern-7mu )}$\label{(3.1)}
\vspace{-1cm}
\begin{align*}
    t_1 &= \prod_{i=0}^{\lfloor \frac{n-2}{2}\rfloor} ( 1 +2i, 2+2i)( n+1 +2i, n+2+2i)\\
    t_2 &= \prod_{i=1}^{\lfloor \frac{n-3}{2}\rfloor} ( 2 +2i, 3+2i)( n+2 +2i, n+3+2i)\\
    t_3 &= (1,3)(2,4)(n+1,n+3)(n+2,n+4)\\
    t_4 &= (1,2)(n+1,n+2)\prod_{i=1}^{n-2}(2+i, n+2+i)
\end{align*}
\end{personalEnvironment}

Observe that $t_1t_2$ when written as a product of pairwise disjoint cycles has two of length 2 and two of length $n-2$. Hence, as $n-2$ is odd, $t_1t_2$ has order $2n-4$. It is easy to check that $t_2t_3$ has order 6 and $t_3t_4$ has order 4. Also we see that $t_1t_3 = t_3t_1$ and $t_2t_4 = t_4t_2$.

Put $G = \langle t_1,t_2,t_3,t_4\rangle$. We will show in (\ref{(3.10)}) 
that $G \cong \B_n$,  after we have first investigated the subgroups $G_{123} = \langle t_1,t_2,t_3\rangle$ and $G_{234} = \langle t_2,t_3,t_4\rangle$.

Beginning with $G_{234}$ and setting 
\begin{align*}
    &\Delta_1 = \{1,2,3,4,5,n+1,n+2,n+3,n+4,n+5\},\\
    &\Delta_6 = \{6,7,n+6,n+7\},\\
    &\Delta_8 = \{8,9,n+8,n+9\},\\
    &\,\,\vdots \\
    &\Delta_{n-1} = \{n-1,n,2n-1,2n\},
\end{align*}

we note that
\vspace{0.5cm}
\begin{personalEnvironment}$\boldsymbol{\mkern-7mu )}$\label{(3.2)}\, the $G_{234}$ orbits of the $\Omega$ are $\Delta_1, \Delta_6, \Delta_8,  \ldots , \Delta_{n-1}$.
\end{personalEnvironment}

We also note 
\vspace{0.5cm}
\begin{personalEnvironment}$\boldsymbol{\mkern-7mu )}$\label{(3.3)} the induced action of $\langle t_2,t_4\rangle$ on each of $\Delta_6,\Delta_8, \ldots , \Delta_{n-1}$ is identical to its action on $\Delta_4 = \{4,5,n+4,n+5\}$.

Set $s_2 = (4,5)(9,10), s_3 = (1,3)(2,4)(6,8)(7,9)$ and $s_4 = (1,2)(3,8)(4,9)(5,10)(6,7)$ (these are $t_2,t_3,t_4$ for the case $n=5$), and $H= \langle s_2,s_3,s_4 \rangle$.
\end{personalEnvironment}
\vspace{0.5cm}
\begin{personalEnvironment}$\boldsymbol{\mkern-7mu )}$\label{(3.4)}
$G_{234} \cong H \cong \mathbb{Z}_2 \times \Sym{5}$ with $\{t_2,t_3,t_4\}$ a C-string for $G_{234}$. Further $\langle (t_2t_3t_4)^5 \rangle = Z(G_{234})$.

Restricting $G_{234}$ to $\Delta_1$ yields a homomorphism from $G_{234}$ to $H$, and then (\ref{(3.3)}) implies $G_{234} \cong H$. Employing \textsc{Magma}\cite{Magma} quickly reveals the structure of $H$ and that $\{s_2,s_3,s_4\}$ is a $C$-string for $H$. This proves (\ref{(3.4)}).
\end{personalEnvironment}

We now turn our attention to $G_{123}$.
\vspace{0.5cm}
\begin{personalEnvironment}$\boldsymbol{\mkern-7mu )}$\label{(3.5)}
\renewcommand{\theenumi}{(\roman{enumi})}
\begin{enumerate}
    \item The $G_{123}$- orbits of $\Omega$ are $\Lambda_1 = \{1,..,n\}$ and $\Lambda_{n+1} = \{n+1,\ldots, 2n\}.$
    \item For $j \in \Lambda_1$ and $g \in G_{123}$, $\, (j)g = k$ if and only if $(j+n)g = k+n$.
\end{enumerate}
\end{personalEnvironment}
\vspace{0.5cm}
\begin{personalEnvironment}$\boldsymbol{\mkern-7mu )}$\label{(3.6)}
For $ 1 \le i < n$, we have $(i,i+1)(n+i,n+i+1) \in G_{123}.$

In view of (\ref{(3.5)})($ii$) it will be sufficient to look at the action of elements of $G_{123}$ on $\Lambda_1$. So, for $i=1,2,3,$ let $\hat{t}_i$ denote the induced action of $t_i$ on $\Lambda_1$. Hence 
\begin{align*}
    &\hat{t}_1 = \prod_{i=0}^{\lfloor \frac{n-2}{2} \rfloor} (1+2i,2+2i),\\
    &\hat{t}_2 = \prod_{i=1}^{\lfloor \frac{n-3}{2} \rfloor} (2+2i,3+2i) \, \text{ and }\\
    &\hat{t}_3 = (1,3)(2,4).
\end{align*}
Consequently
$$\hat{t}_1\hat{t}_2 = (1,2)(3,5,7,\ldots,n,n-1,n-3,\ldots,6,4).$$
Since $n-2$ is odd, 
$(\hat{t}_1\hat{t}_2)^{n-2} = (1,2)$. Therefore, $(1,2)(n+1,n+2) \in G_{123}$. Also 
\begin{align*}
    \hat{t}_3\hat{t}_1\hat{t}_2 &= (1,3)(2,4)(1,2)(3,5,7,\ldots,n,n-1,n-3,\ldots,6,4)\\
    &= (1,5,7,\ldots,n,n-1,n-3,\ldots,6,4)(2,3)
\end{align*}
which is in disjoint cycle form. Again, since $n-2$ is odd, we have $(\hat{t}_3\hat{t}_1\hat{t}_2)^{n-2} = (2,3)$. Hence $(2,3)(n+2,n+3) \in G_{123}.$ Now $(3,4)  = (1,2)^{\hat{t}_3}$, so $(3,4)(n+3,n+4) \in G_{123}.$

Now we recursively construct the remaining $(i,i+1)(n+i,n+i+1)$ for all $i$, with $3 < i < n$. Supposing we have $(i,i+1)(n+i,n+i+1) \in G_{123}$ for all $3 \le i \le k$. We show that $(k+1,k+2)(n+k+1,n+k+2) \in G_{123}$. If $k$ is even, then 

\begin{align*}
    (k,k+1)(n+k,n+k+1)^{t_1} &= 
    (k-1,k+2)(n+k-1,n+k+2).
\end{align*}
Since $$
    (k-1,k+2)(n+k-1,n+k+2)^{(k-1,k+1)(n+k-1,n+k+1)} =
     (k+1,k+2)(n+k+1,n+k+2)
$$ and $$(k-1,k+1)(n+k-1,n+k+1) = (k-1,k)(n+k-1,n+k)^{(k,k+1)(n+k,n+k+1)} \in G_{123}, $$
we deduce that $(k+1,k+2)(n+k+1,n+k+2) \in G_{123}$. When $k$ is odd, a similar calculation using $t_2$ in place of $t_1$, also yields the same conclusion, so proving (\ref{(3.6)}).
\end{personalEnvironment}
\vspace{0.5cm}
\begin{personalEnvironment}$\boldsymbol{\mkern-7mu )}$\label{(3.7)}
$G_{123} \cong \Sym{n}$. 

From (\ref{(3.6)})
$$K = \langle \,\, (i,i+1)(n+i,n+i+1)  \,\,| \,\,1 \le i < n  \,\,\rangle \le G_{123}, $$
with the generators of $K$ satisfying the Coxeter relations for $\Sym{n}$. Thus, by Lemma \ref{lemma2.2}, $K$ is isomorphic to a quotient of $\Sym{n}$ and hence $K \cong \Sym{n}$. The action of $G_{123}$ on $\{\{i,n+i\} | i \in \Lambda_1\}$ forces $K=G_{123}$, so giving (\ref{(3.7)}).
\end{personalEnvironment}
\vspace{0.5cm}
\begin{personalEnvironment}$\boldsymbol{\mkern-7mu )}$\label{(3.8)}
$\{t_1,t_2,t_3\}$ is a C-string for $G_{123}$.

We only need check $\langle t_1,t_2 \rangle \cap \langle t_2,t_3 \rangle = \langle t_2 \rangle$, the other intersections being clear. Now $\langle t_2t_3 \rangle$ has $\{1,3\}$ and $\{2,4,5\}$ as an orbits on $\Omega$ and so  $\langle (t_2t_3)^2 \rangle$ has $\{2,4,5\}$ as an orbit and $\langle(t_2t_3)^3\rangle$ has $\{1,3\}$. Since $\langle t_1t_2 \rangle$ has $\{1,2\}$ as an orbit and $t_2t_3$  has order 6, we conclude that $\langle t_1,t_2\rangle \cap \langle t_2,t_3\rangle = \langle t_2\rangle.$ So (\ref{(3.8)}) holds. 
\end{personalEnvironment}

Set $\omega_0 = \prod_{i=1}^n(i,n+i)$.
\vspace{0.5cm}
\begin{personalEnvironment}$\boldsymbol{\mkern-7mu )}$\label{(3.9)}
$\omega_0 = (t_1t_2t_3t_4)^n.$

We calculate that 

\hspace{-2cm}\begin{align*}
    t_1t_2 &= (1,2)(3,5,7,\ldots,n,n-1,\ldots,6,4)(n+1,n+2)(n+3, \ldots , 2n, 2n-1,\ldots n+4),\\
    t_1t_2t_3 &= t_1t_2(1,3)(2,4)(n+1,n+3)(n+2,n+4)
    \\
    &= (1,4)(3,5,\ldots,n,n-1,\ldots,8,6,2)(n+1,n+4)(n+3, \ldots , 2n,2n-1, \ldots, n+4), \\
        t_1t_2t_3t_4 &= t_1t_2t_3 (1,2)(n+1,n+2)\prod_{i=1}^{n-2}(2+i,(n+2)+i)
    \\
     &= (\underbrace{1,n+4,n+2,3,n+5,7,\ldots,n,2n-1,n-3,2n-5,\ldots,n+6,n+1}_{n+1,}\\
     &\underbrace{4,2,n+3,\ldots,6}_{n-1})\\ 
     &\text{when  } n \equiv 1 \mod 3 \text{ and }\\
    &=(\underbrace{1,n+4,n+2,3,n+5,7,\ldots,2n,n-1,2n-3,n-5,\ldots,n+6,n+1}_{n+1},\\
    &\underbrace{4,2,n+3,\ldots,6}_{n-1})\\  &\text{when  } n \equiv 3 \mod 3. 
    \end{align*}
    
    Therefore $t_1t_2t_3t_4 = \prod_{i=1}^{n}(i,n+i)$.
\end{personalEnvironment}    
\vspace{0.5cm}
\begin{personalEnvironment}$\boldsymbol{\mkern-7mu )}$\label{(3.10)}
$G \cong \B_n$. 

Set $\beta_0 = (1,n+1)$ and, for $1 \le i < n$,  $\beta_i = (i,i+1)(n+i,n+i+1)$. Put $L = \langle \beta_0,\beta_1,\ldots,\beta_{n-1}\rangle$. 
By Lemma \ref{lemma2.2} $L \cong \B_n$.

Directly from their definitions, we have 
\vspace{-0.1cm}
\begin{align*}
    t_1 &= \prod_{i=0}^{\lfloor \frac{n-2}{2} \rfloor} \beta_{1+2i} \text{ and }\\
    t_2 &= \prod_{i=0}^{\lfloor \frac{n-3}{2} \rfloor} \beta_{2+2i}. 
\end{align*}
An easy check shows $t_3 = {\beta_1}^{\beta_2}{\beta_2}^{\beta_3}$. Since 
\begin{align*}
    t_4 &= (1,2)(n+1,n+2)\prod_{i=1}^{n-2}(2+i,n+2+i)\\
    &= \beta_1\prod_{i=2}^n \beta_0^{\prod_{j=i}^i \beta_j},
\end{align*}
we infer that $G \le L$. 

By (\ref{(3.6)}) $\beta_i \in G$ for $1 \le i < n.$ Thus to complete the proof of (\ref{(3.10)}) we need to demonstrate that $\beta_0 \in G$. Now
\begin{align*}
    {t_2}^{t_3} &= (\prod_{i=1}^{\lfloor \frac{n-3}{2}\rfloor}(2+2i,3+2i)(n+2+2i,n+3+2i))^{(1,3)(2,4)(n+1,n+3)(n+2,n+4)}\\
    &= (2,5)(n+2,n+5)t_2(4,5)(n+4,n+5)  \\
     \intertext{and} 
     {t_2}^{t_3t_4} &= (1,n+5)(n+1,5)t_2(4,5)(n+4,n+5).\\
     \intertext{Hence} 
     {t_2}^{t_3}{t_2}^{t_3t_4} &= (2,5)(n+2,n+5)t_2(4,5)(n+4,n+5)(1,n+5)(n+1,5)t_2(4,5)(n+4,n+5)\\
    &= (2,5)(n+2,n+5)(1,n+5)(n+1,5)\\
    &= (2,n+1,5)(1,n+5,n+2).\\
\intertext{Therefore} 
     t_1{t_2}^{t_3}{t_2}^{t_3t_4} &= t_1t_2(2,n+1,5)(1,n+5,n+2)\\
    &=(1,n+1)(2,n+5,n+7,\ldots,2n,2n-1,\ldots,n+4,n+3, \\
    & \quad\quad n+2,5,7,9,\ldots,n,n-1,\ldots,6,4,3),\\
    \intertext{and so}  (t_1{t_2}^{t_3}{t_2}^{t_3t_4})^{n-1} &= \prod_{i=2}^n (i,n+i),
\end{align*}
Hence, using (\ref{(3.9)}), 
$$ \beta_0 = \prod^n_{i=1}(i,n+i) (t_1{t_2}^{t_3}{t_2}^{t_3t_4})^{-n+1} \in G,$$ which proves $(\ref{(3.10)})$.
\end{personalEnvironment}
\vspace{0.5cm}
\begin{personalEnvironment}$\boldsymbol{\mkern-7mu )}$\label{(3.11)}
$\{t_1,t_2,t_3,t_4\}$ is a string C-string for $G$. 

From (\ref{(3.4)}) $G_{234}\cong \mathbb{Z}_2 \times \Sym{5}$ and $\langle (t_2t_3t_4)^5 \rangle = Z(G_{234})$. If $G_{123} \cap G_{234} > G_{23}$, then, as $G_{23} \cong \Dih{12}$, we must have either $(t_2t_3t_4)^5 \in G_{123}\cap G_{234}$ or $G_{123} \cap G_{234}$ has index at most $2$ in $G_{234}$ (and so $t_3t_4 \in G_{123}$). Either of these possibilities would contradict (\ref{(3.5)})($i$) as $(t_2t_3t_4)^5 : 1 \rightarrow n+1$ and $t_3t_4 : 5 \rightarrow n + 5$. Thus $G_{123} \cap G_{234} = G_{23}$. Using (\ref{(3.4)}) and (\ref{(3.8)}) we now obtain (\ref{(3.11)}).
\end{personalEnvironment}

\begin{personalEnvironment}$\boldsymbol{\mkern-7mu )}$\label{(3.12)}
When $n \ge 7$, $\{t_1,t_2,t_3,t_4\}$ is an unravelled C-string for $G$.

Let $M_1 = \langle \,\,(i,n+i) \,\,|\,\, 1 \le i \le n \,\,\rangle (=\O_2(G))$ and $M_2=\langle \,\, (1,n+1)(i,n+i) \,\,|\,\, 1 < i \le n \,\,\rangle$. From (\ref{(3.10)}) $G \cong \B_n$ and hence for $N \trianglelefteq G$, $1 \ne N \ne G$, we either have $[G:N] \le 4$ or $N = \langle \omega_0 \rangle,$ $M_1$ or $M_2$. Set $\overline{G} = G/N$. Since $\{t_1,t_2,t_3,t_4\}$ has rank 4, we are only required to check that $\{\overline{t}_1,\overline{t}_2,\overline{t}_3,\overline{t}_4\}$ is not a C-string for $N = \langle \omega_0 \rangle, M_1$ and $M_2$. 

Suppose $N=M_1$ or $M_2$. Then 
$$g = (t_3t_4)^2=(1,n+2)(2,n+1)(3,n+4)(4,n+3) \in G_{234}.$$

Also, using (\ref{(3.6)}), 
$$h = (1,2)(3,4)(n+1,n+2)(n+3,n+4) \in G_{123}. $$

Since $$ gh^{-1} = (1,n+1)(2,n+2)(3,n+3)(4,n+4) \in M_2 \le M_1, $$ 
we get $\overline{g} = \overline{h} \in \overline{G}_{123} \cap \overline{G}_{234}.$ The non-trivial elements of $M_1$ either fix an element of $\Lambda_1$ or maps it to an element of $\Lambda_2$. Hence, as $\{1,3\}$ is a $G_{23}$-orbit, $\overline{h} \notin \overline{G}_{23}$. Thus $\overline{G}_{123} \cap \overline{G}_{234} > \overline{G}_{23}$ when $N=M_1$ or $M_2$. Now suppose $N = \langle \omega_0 \rangle$. This time we take 
$$ g = \prod^5_{i=1}(i,n+i)\prod_{j=1}^{\lfloor \frac{n-5}{2}\rfloor}(4+2j,n+5+2j)(5+2j,n+4+2j) \text{ and }$$
$$ h = \prod_{j=1}^{\lfloor \frac{n-5}{2}\rfloor}(4+2j,5+2j)(n+4+2j,n+5+2j).$$

Then $g \in G_{234}$, $h \in G_{123}$ and $gh^{-1} = \omega_0.$ Therefore $\overline{g} = \overline{h} \in \overline{G}_{123} \cap \overline{G}_{234}.$ It is straightforward to also see that $\overline{g} \notin \overline{G}_{23},$ and consequently (\ref{(3.12)}) is proven.
\end{personalEnvironment}

Combining (\ref{(3.10)}),(\ref{(3.11)}) and (\ref{(3.12)}) completes the proof of Theorem \ref{thm1.2}. 

\section{Rank $n-4$ Unravelled C-Strings}

This final section is devoted to the proof of Theorem \ref{thm1.3}. Thus we assume $n\ge8$ and we set $m = n-4$.

Just as in the proof of Theorem \ref{thm1.2} we construct $\{t_1,t_2,\ldots,t_m\}$ as a subset of $\Sym{2n}$ and then show that it is a C-string for $\B_n$. Finally, when $n$ is even, we prove that it is an unravelled C-string. So again, let $\Omega = \{1,\ldots,2n\}$ and define the $t_i$ as follows. 

\begin{personalEnvironment}$\boldsymbol{\mkern-7mu )}$\label{(4.1)}
\vspace{-0.5cm}
\begin{align*}
    t_1 &= (2,3)(n+2,n+3)(4,5)(n+4,n+5)\prod_{i=6}^n(i,n+i)\\
    t_2 &= (1,2)(n+1,n+2)(3,4)(n+3,n+4)(5,6)(n+5,n+6)\prod_{i=7}^n(i,n+i)\\
    t_3 &= (2,3)(n+2,n+3)(6,7)(n+6,n+7)\\ \intertext{and for $k = 4,\ldots,m$,}
    t_k &= (k+3,k+4)(n+k+3,n+k+4).
\end{align*}
\end{personalEnvironment}
Set $I = \{1,2,\ldots,m\}.$ 

\begin{personalEnvironment}$\boldsymbol{\mkern-7mu )}$\label{(4.2)} $t_3t_4 = (2,3)(n+2,n+3)(6,8,7)(n+6,n+8,n+7).$
\end{personalEnvironment}

Next we show that 

\begin{personalEnvironment}$\boldsymbol{\mkern-7mu )}$\label{(4.3)}
$(t_1t_2\ldots t_mt_3t_4)^n = \prod_{i=1}^n(i,n+i).$
\begin{align*}
    \intertext{We calculate that}
    t_4t_5\ldots t_m =& (7,n,n-1,n-2,\ldots,9,8)(n+7,2n,2n-1,\ldots,n+8)\\ \intertext{and}
    t_1t_2t_3 =& (1,3)(n+1,n+3)(2,4,7,6,n+5,n+2,n+4,n+7,n+6,5).\\
    \intertext{Hence}
    t_1t_2\ldots t_m =& (1,3)(n+1,n+3)(2,4,n,n-1,\ldots,9,8,7,6,n+5,\\&n+2,n+4,2n,2n-1,\ldots,n+8,n+7,n+6,5).\\
    \intertext{Using (\ref{(4.2)}) we now get }
    t_1t_2\ldots t_mt_3t_4 =& (1,2,4,n,n-1,\ldots,9,7,8,6,n+5,\\&n+3,n+1,n+2,n+4,2n,\ldots,n+9,n+7,n+8,n+6,5,3),\\
    \intertext{which yields (\ref{(4.3)}).}
\end{align*}
\end{personalEnvironment}

\begin{personalEnvironment}$\boldsymbol{\mkern-7mu )}$\label{(4.4)} $(t_1t_2\ldots t_{m-1}t_3t_4)^{n-1} = \prod_{i=1}^{n-1}(i,n+i).$

\begin{align*}
   \intertext{First we have} 
   t_1t_2 \ldots t_{m-1} =& (1,3)(n+1,n+3)(2,4,n-1,n-2,\ldots,6,\\&n+5,n+2,n+4,2n-1,\ldots,n+7,n+6,5),\\
   \intertext{then, using (\ref{(4.2)}),}
   t_1t_2\ldots t_{m-1}t_3t_4 =& (1,2,4,n-1,\ldots,9,7,8,6,n+5,n+3,\\&n+1,n+2,n+4,2n-1,\ldots,n+9,n+7,n+8,n+6,5,3).\\
   \intertext{This gives the desired expression for $(t_1t_2\ldots t_{m-1}t_3t_4)^{n-1}$.}
\end{align*}
\end{personalEnvironment}

Combining (\ref{(4.3)}) and (\ref{(4.4)}) we observe that 

\begin{personalEnvironment}$\boldsymbol{\mkern-7mu )}$\label{(4.5)} $(t_1t_2\ldots t_{m-1}t_3t_4)^{n-1}(t_1t_2\ldots t_{m}t_3t_4)^{n} = (n,2n)$.

\end{personalEnvironment}

\begin{personalEnvironment}$\boldsymbol{\mkern-7mu )}$\label{(4.6)}
For $i,j \in \{1,2,\ldots,m\}$, the order of $$t_it_j \text{ is } \begin{dcases} 1 &\text{ if } i = j \\ 2 &\text{ if } |i -j| \ge 2 \\
12  &\text{ if } i=1,j=2 \text{ or } i=2,j=3\\ 6 &\text{ if } i=3,j=4\\ 3 &\text{ otherwise.}\end{dcases} $$
 It is evident that each $t_i$ is an involution as they are defined as the products of pairwise disjoint transpositions. Since
 \begin{align*}
     t_1t_2 =& (1,2,4,6,n+5,n+3,n+1,n+2,n+4,n+6,5,3),
     \intertext{$t_1t_2$ has order 12. Similarly we have }
     t_2t_3 =& (1,3,4,2)(n+1,n+3,n+4,n+2)(5,7,n+6,n+5,n+7,6)\prod_{i=8}^n(i,n+i),\\ \intertext{
 and so $t_2t_3$ also has the order 12. From (\ref{(4.2)}) we see that order of $t_3t_4$ is 6. If, $|i-j| = 1$ and $4 \le i < j \le m$, then}
 t_it_j =& (3+i,5+i,4+i)(n+3+i,n+5+i,n+4+i) 
  \end{align*}
 has order 3. That $t_i$ and $t_j$ commute when $|i-j| \ge 2$ is readily checked, so verifying (\ref{(4.6)}). 
\end{personalEnvironment}

Put $G = \langle t_1,t_2,\ldots,t_m\rangle.$

\begin{personalEnvironment}$\boldsymbol{\mkern-7mu )}$\label{(4.7)}
$G\cong \B_n$.

We again employ Lemma \ref{lemma2.1} to identify $G$. So set  $\beta_0 = (1,n+1),$ $\beta_i = (i,i+1)(n+i,n+i+1)$, for $1 \le i < n,$ and $L = \langle \beta_0,\beta_1,\ldots,\beta_{n-1}\rangle \le \Sym{2n}$. Also set $\eta_i = (i,n+i)$ for $1 \le i \le n$. Note that $\eta_1 = \beta_0$ and $\eta_i = {\beta_0}^{\prod_{j=1}^{i-1}\beta_j}$ for $i=2,\ldots,n.$ Therefore $\eta_i \in L$ for $i =1,\ldots,n.$ Because
\begin{align*}
    t_1 =& \beta_2\beta_4\prod_{i=6}^{n}
\eta_i,\\
    t_2 =&\beta_1\beta_3\beta_5 \prod_{i=7}^n\eta_i,\\
    t_3 =& \beta_2\beta_6 \,\, \text{ and, for $4 \le i \le m$},\\
    t_i =& \beta_{i+3}
\end{align*}
we conclude that $G \le L.$

From (\ref{(4.5)}), $\eta_n = (n,2n) \in G$. Now let $g = t_mt_{m-1}\ldots t_4t_3t_2t_1t_2t_1t_2 \in G.$ Then we see that ${\eta_n}^g = \eta_1 = \beta_0,$ whence $\beta_0 \in G.$ Since $\beta_i = t_{i-3}$ for $i=7,\ldots,n-1,$ it remains to show that $\beta_1,\beta_2,\beta_3,\beta_4,\beta_5$ and $\beta_6$ are in $G$.

Employing (\ref{(4.2)}) again we have 
\begin{align*}
    \beta_2 =& (t_3t_4)^3,\\
    \beta_6 =& t_3(t_3t_4)^3,\\
    \beta_1 =& {\beta_6}^{t_2t_3t_1t_2t_1t_2t_3} \\
    \beta_3 =& {\beta_6}^{t_2t_3t_1t_2t_1t_3}\\
    \beta_4 =& {\beta_6}^{t_2t_1t_2t_1t_3t_2} \text{ and}\\
    \beta_5 =& {\beta_6}^{t_2t_3\eta_6}.
\end{align*}

Since $\eta_6 = {\eta_{n}}^h$ where $h = t_mt_{m-1}\ldots t_4 t_3$, we have now shown that $\beta_i \in G$ for $i=0,\ldots,n-1.$ Thus $G=L,$ and (\ref{(4.7)}) is proven.
\end{personalEnvironment}
We now turn our attention to showing that $\{t_1,\ldots,t_m\}$ is a C-string.

For $t_1,\ldots,t_m$ if we wish to highlight that they are permutations in $\Sym{2n}$ we shall write $t_1^{(n)},\ldots,{t_m}^{(n)}.$ Set $G_{1234}^{(n)} = \langle t_1^{(n)},t_2^{(n)},t_3^{(n)},t_4^{(n)} \rangle$, $G_{123}^{(n)} = \langle t_1^{(n)},t_2^{(n)},t_3^{(n)}\rangle$, $G_{234}^{(n)} = \langle t_2^{(n)},t_3^{(n)},t_4^{(n)} \rangle$ and $G_{23}^{(n)} = \langle t_2^{(n)},t_3^{(n)}\rangle$.
\begin{personalEnvironment}$\boldsymbol{\mkern-7mu )}$\label{(4.8)}
For $n\ge8$, $\{ t_1^{(n)},t_2^{(n)},t_3^{(n)},t_4^{(n)} \}$ is a C-string for $G_{1234}^{(n)}$.

First we may verify (\ref{(4.8)}) for $n=8$ using \textsc{Magma}. Then we may define $\mu \in \Sym{2n}$ by \begin{align*}
    \mu &: \omega \rightarrow \omega-n+8 \text{ for $\omega \in \Lambda =\{n+1,n+2,n+3,n+4,n+5,n+6,n+7,n+8\}$} \\
    &: \omega \rightarrow \omega \text{ for $\omega \in \Omega\setminus \Lambda$}.
\end{align*}

For $i=1,2,3,4$ define $$\phi(t_i^{(n)}) = \widehat{(t_i^{(n)})^\mu} $$
where $\,\widehat{  \quad\quad}\,$ denotes the induced action upon the set $\Phi = \{1,\ldots ,16\}.$ When we write equalities in this context, it is as a permutation of $\Phi$. Observe that $\phi$ extends to a homomorphism from $G_{1234}^{(n)}$ to $\Sym{\Phi}$ with $\phi(G_{1234}^{(n)}=G_{1234}^{(8)}$.

Because 
\begin{align*}
    \phi(G_{23}^{(n)}) \le \phi(G_{123}^{(n)} \cap G_{234}^{(n)}) \le\phi(G_{123}^{(n)}) \cap \phi(G_{234}^{(n)}) 
    =G_{123}^{(n)} \cap G_{234}^{(n)} 
\end{align*} 
and (\ref{(4.8)}) holds for $n=8$, we have $$\phi(G_{123}^{(n)} \cap G_{234}^{(n)}) = G_{23}^{(8)}.$$

We now investigate the structure of $H=G_{234}^{(n)}.$ Set $s_2 = t_2^{(n)}, s_3 = t_3^{(n)}, s_4 = t_4^{(n)}$ and $R= \langle (s_2s_3)^3,(s_3s_4)^3\rangle$. Now $(s_3s_4)^3$ inverts $(s_2s_3)^3$ which has order 4. Therefore $R \cong \Dih{8}.$ Calculation shows that $s_2,s_3$ and $s_4$ normalize $R$ and hence $R\trianglelefteq H$. Set $C=C_H(R)$. Further calculation shows that $s_2,s_3,s_2s_3 \notin C$ but $s_4,(s_2s_3)^2 \in C.$ Therefore, as $H = \langle s_2,s_3,s_4 \rangle$, $H=\langle s_2,s_3\rangle C$ with $H/C \cong 2^2.$ Also we have $(s_2s_3)^6 \in Z(H)$ with $s_4(s_2s_3)^4$ of order 4. Thus, as  $(s_2s_3)^4$ has order 3 and $s_4$ has order 2, $\langle (s_2s_3)^4,s_4 \rangle \cong \Sym{4}.$ Thus $S = \langle (s_2s_3)^2,s_4\rangle \cong 2 \times \Sym{4}$ with $S\le C.$ Since $s_2$ and $s_3$ normalize $S$, we infer $C=S.$ In particular, we have shown $G_{234}^{(n)} = H$ has order $2^6.3$ for all $n\ge 8.$ 

Consequently $\phi$ restricted to $G_{234}^{(n)} \rightarrow G_{234}^{(8)}$ is an isomorphism. So calculation in $G_{234}^{(n)}$ may be performed in $G_{234}^{(8)}$ and, using $\phi$, we may keep track of the action on $\Omega$. 

Now $G_{234}^{(n)}$ has orbit $\{5,6,7,8,n+5,n+6,n+7,n+8\}$ on $\Omega$ and $G_{123}^{(n)}$ has $\{8,n+8\}$ as an orbit. Thus
$$G_{23}^{(n)} \le G_{123}^{(n)} \cap G_{234}^{(n)} \le Stab({5,6,7,n+5,n+6,n+7\}})=T^{(n)}. $$

Calculation shows that $T^{(n)} = \langle G_{123}^{(n)},(6,7)(n+6,n+7)\rangle$ with $[T^{(n)}:G_{23}^{(n)}]=2$. If $G_{23}^{(n)} < G_{123}^{(n)} \cap G_{234}^{(n)},$ then $G_{123}^{(n)} \cap G_{234}^{(n)} = T^{(n)}$ must contain a normal subgroup of order 2 intersecting $G_{23}^{(n)}$ trivially (the kernel of $\phi$, restricted to $G_{123}^{(n)} \cap G_{234}^{(n)}$), but it does not. Thus $G_{23}^{(n)} = G_{123}^{(n)} \cap G_{234}^{(n)}.$ That $\{t_2^{(n)},t_3^{(n)},t_4^{(n)}\}$ is a C-string for $G_{234}^{(n)}$ follows from $G_{234}^{(n)}$ and $G_{234}^{(8)}$ being isomorphic and the fact that $\phi$ maps generator to generator. Observe that the $G_{23}-$orbits of $\Omega$ are $\{1,2,3,4\}, \{5,6,7,n+5,n+6,n+7\}, \{n+1,n+2,n+3,n+4\}, \{i,n+i\} (i=8,\ldots,n).$ If $G_{12} \cap G_{23} > G_2$ then one of 
\begin{align*}
    (t_1t_2)^4 &= (1,n+5,n+4)(2,n+3,n+6)(4,n+1,5)(6,n+2,3) \text{ and }\\
    (t_2t_3)^6 &= (1,n+1)(2,n+2)(3,n+3)(4,n+4)(5,n+5)(6,n+6)
\end{align*}
would be in $G_{23}$. But then $1$ and $n+1$ would be in the same $G_{23}-$orbit, a contradiction. Therefore $G_{12} \cap G_{23} = G_2.$
Appealing to Lemma \ref{lemma2.3}  this now proves (\ref{(4.8)}).

\end{personalEnvironment}

\begin{personalEnvironment}$\boldsymbol{\mkern-7mu )}$\label{(4.9)} Let $4 \le k \le m$ and set $J_k = I_k\setminus\{1,2\}$. Then \begin{enumerate}[$(i)$]
    \item $G_{J_k} \cong \Sym{k-2}$ and
    \item $\{t_4,\ldots,t_k\}$ is a C-string for $G_{J_k}$.
\end{enumerate}

Since $G \cong \B_n$ by (\ref{(4.7)}), $\langle t_4,\ldots , t_k \rangle = \langle \beta_1,\ldots, \beta_{k+3}\rangle$ is a standard parabolic subgroup of $G$. Hence (\ref{(4.9)})$(i)$ and $(ii)$ follow.
\end{personalEnvironment}

\begin{personalEnvironment}$\boldsymbol{\mkern-7mu )}$\label{(4.10)}
Let $4 \le k \le m$ and set $J_k = I_k\setminus\{1,2\}$.
Then \begin{enumerate}[$(i)$]
    \item $G_{J_k} \equiv \mathbb{Z}_2 \times \Sym{k-1};$ and
    \item $\{t_3,\ldots,t_k\}$ is a C-string for $G_{J_k}$.
\end{enumerate}
Recall that $t_3 = \beta_2\beta_6$, and that $t_i = \beta_{i+3}$ for $4 \le i \le k$. Since $(\beta_2\beta_6\beta_7)^3 = \beta_2$, we have 
\begin{align*}
    \langle t_3,\ldots,t_k\rangle =& \langle \beta_2\beta_6,\beta_7\ldots,\beta_{k+3}\rangle \\
    =& \langle \beta_2,\beta_6,\beta_7\ldots,\beta_{k+3}\rangle\\
    \cong& \mathbb{Z}_2 \times \Sym{k-1},
\end{align*}
so giving part $(i)$. Further, as $\langle t_3,\ldots,t_k\rangle$ is a Coxeter group,
\begin{align*}
    \langle t_3,t_4,\ldots,t_{k-1} \rangle \cap \langle t_4,\ldots t_k\rangle &=
    \langle \beta_2,\beta_6,\beta_7,\ldots,\beta_{k+2} \rangle \cap \langle \beta_7,\ldots \beta_{k+3}\rangle\\ &=
    \langle\beta_7,\ldots,\beta_{k+2}\rangle\\ &= \langle t_4,\ldots,t_{k-1} \rangle.
\end{align*}
To prove $(ii)$ we may argue by induction on $k$, $k=4$ being covered by (\ref{(4.8)}). Thus $\{t_3,t_4,\ldots,t_{k-1}\}$ is a C-string for $G_{\{3,\ldots,k-1\}}$ and, as $\{t_4,\ldots,t_k\}$ is a C-string for $G_{\{4,\ldots,k\}}$, Lemma \ref{lemma2.3} yields $(ii)$.
    
\end{personalEnvironment}

Set $t_0 = {t_4}^{(t_3t_4t_2t_3)}t_2.$

\begin{personalEnvironment}$\boldsymbol{\mkern-7mu )}$\label{(4.11)}
For $4 \le k \le m,$ set $J_k = I_k\setminus\{1\}.$ Then
\begin{align*}
    G_{J_k} =& \langle t_0,\beta_2\rangle \times \langle \beta_5,\beta_6,\ldots,\beta_{k+3}\rangle\\
    \cong& \Dih{8} \times \Sym{k}.
\end{align*}
Clearly $t_0 \in G_J$ and calculation reveals that
    $$t_0 = (1,2)(n+1,n+2)(3,4)(n+3,n+4)\prod_{i=7}^n(i,n+i).$$
Hence $t_0\beta_5 = t_2$.

Set $H = \langle t_0,\beta_2,\beta_5,\beta_6,\ldots,\beta_{k+3}\rangle$. Observing that $t_0$ and $\beta_2$ commute with each of $\beta_5,\beta_6,\ldots,\beta_{k+3},$ we have
$$H = \langle t_0,\beta_2\rangle \times \langle \beta_5,\beta_6,\ldots,\beta_{k+3}\rangle. $$
We show that $G_{J_k} = H.$ Recalling that $\beta_{i+3} = t_i$, $ 4 \le i \le m$ and $t_3 = \beta_2\beta_6$, $t_2 = t_0\beta_5$ implies $G_{J_k} \le H$. Since $t_0 = t_2\beta_5$, $\beta_5 = t_0t_2, \beta_6 = t_3(t_3t_4)^3$ and $\beta_2 = t_3\beta_6$, we also have $H \le G_{J_k}$. Because $t_0\beta_2$ has order 4 and $\langle \beta_2,\beta_6,\ldots,\beta_{k+3}\rangle$ is a standard parabolic subgroup of $\B_n$, we deduce that $G_{J_k} \cong \Dih{8} \times \Sym{k}$.
\end{personalEnvironment}

\begin{personalEnvironment}$\boldsymbol{\mkern-7mu )}$\label{(4.12)}
For $J_k=I_k\setminus\{1\}$ where $4 \le k \le m$, $\{t_2,\ldots,t_k\}$ is a C-string for $G_{J_k}$.

We argue by induction on $k$. By our induction hypothesis we have that $\{t_2,t_3,\ldots,t_{k-1}\}$ is a C-string for $G_{{J_k}\setminus\{k\}}$. From (\ref{(4.10)}) $\{t_3,\ldots,t_k\}$ is a C-string for $G_{{J_k}\setminus\{2\}}$. Also, from (\ref{(4.10)}),
\begin{align*}
    G_{{J_k}\setminus\{2,k\}} &\cong \mathbb{Z}_2 \times \Sym{n-6} \,\,\,\text{ and }\\
    G_{{J_k}\setminus\{k\}} &\cong \mathbb{Z}_2 \times \Sym{n-5}.
\end{align*}
Hence $G_{{J_k}\setminus\{2,k\}}$ is a maximal subgroup of $G_{{J_k}\setminus\{k\}}$. So, if $G_{{J_k}\setminus\{k\}} \cap G_{{J_k}\setminus\{2\}} > G_{{J_k}\setminus\{2,k\}}$, then $G_{{J_k}\setminus\{k\}} \le G_{{J_k}\setminus\{2\}}$ which means $t_2 \in G_{{J_k}\setminus\{2\}}$. But $G_{{J_k}\setminus\{2\}}$ fixes 1 whereas $t_2$ does not, a contradiction. Therefore $G_{{J_k}\setminus\{k\}} \cap G_{{J_k}\setminus\{2\}} = G_{{J_k}\setminus\{2,k\}}$. Thus, using Lemma \ref{lemma2.3}, we get (\ref{(4.12)}).
\end{personalEnvironment}

\begin{personalEnvironment}$\boldsymbol{\mkern-7mu )}$\label{(4.13)}
For $4 \le k \le m,$ $\{t_1,t_2,\ldots,t_k\}$ is a C-string for $G_{I_k}.$

We again argue by induction on $k$, with (\ref{(4.8)}) starting the induction. So $k > 4$ and $\{t_1,t_2,\ldots,t_{k-1}\}$ is a C-string for $G_{I_{k-1}}$. By (\ref{(4.12)}) we have that $\{t_2,\ldots,t_{k-1}\}$ is a C-string for $G_{I_{k-1}\setminus\{1\}}$. Then, using (\ref{(4.11)}) we have that
\begin{align*}
    G_{{I_k}\setminus\{1\}} &\cong \Dih{8} \times \Sym{k} \text{\quad and}\\
    G_{I_{k-1}\setminus\{1\}} &\cong \Dih{8} \times \Sym{k-1}.
\end{align*}

Hence $G_{I_{k-1}\setminus\{1\}}$ is a maximal subgroup of $G_{{I_k}\setminus\{1\}}$. So, if $G_{{I_k}\setminus\{1\}} \cap G_{{I_{k-1}}} > G_{{I_{k-1}}\setminus\{1\}}$, then $G_{{I_k}\setminus\{1\}} \le G_{I_{k-1}}$. But then $t_k \in G_{I_{k-1}},$ which is not the case as $t_k$ moves $k+3$ to $k+4$ which are in different $G_{{I_{k-1}}}$-orbits. Therefore $G_{{I_k}\setminus\{1\}} \cap G_{{I_{k-1}}} = G_{{I_{k-1}}\setminus\{1\}}$. Thus, using Lemma \ref{lemma2.3}, we get (\ref{(4.13)}).
\end{personalEnvironment}

\begin{personalEnvironment}$\boldsymbol{\mkern-7mu )}$\label{(4.14)}
 $\{t_1,t_2,\ldots,t_k\}$ is a C-string for $G.$
 
 Taking $m=k$ in (\ref{(4.13)}) gives (\ref{(4.14)}).
\end{personalEnvironment}

\begin{personalEnvironment}$\boldsymbol{\mkern-7mu )}$\label{(4.15)}
If $n$ is even, then  $\{t_1,\ldots,t_m\}$ is an unravelled C-string.

Set $w_0 = \prod_{i=1}^{n}(i,n+i)$, $g = (t_3(t_3t_4))^3 (= \beta_6)$ and $$h = (t_1t_2)^2t_1(t_3t_2t_1)^3t_3^{t_2}(t_1t_3t_2)^4.$$ Then $$h=\prod_{i=1}^5(i,n+i)(6,7,n+6,n+7)\prod_{i=8}^n(i,n+i).$$Hence $gh = w_0$, and $g \in G_{\{3,4\}}, h \in G_{\{1,2,3\}}$. 

Put $M_1 = \langle (i,n+i) | 1 \le i \le n \rangle$ and $M_2 = \langle (1,n+1)(i,n+i) | 1 < i \le n \rangle.$ Since $n$ is even, $\langle w_0 \rangle \le M_2 \le M_1$. Let $\overline{G} = G/N$ where $N$ is one of $\langle w_0 \rangle, M_2, M_1.$ Then $\overline{g} \overline{h^{-1}} \in \overline{G}_{\{1,2,3\}} \cap \overline{G}_{\{3,4\}}$ with $\overline{g} \ne \overline{t_3}$ whence $\overline{G}_{\{1,2,3\}} \cap \overline{G}_{\{3,4\}} > \overline{G_{\{3\}}}$, which proves (\ref{(4.14)}).
\end{personalEnvironment}

Together, (\ref{(4.14)}) and (\ref{(4.15)}) establish Theorem \ref{thm1.3}.

\end{document}